\documentclass[11pt]{amsart}
\usepackage{amssymb,courier}
\input amssym.def
\usepackage{amsmath, amsfonts,hyperref}
\usepackage{amscd,gensymb}
\usepackage[mathscr]{eucal}
\usepackage{graphicx}
\usepackage{float}

\usepackage{xcolor}
\hypersetup{
    colorlinks,
    linkcolor={red!50!black},
    citecolor={blue!50!black},
    urlcolor={blue!80!black}
}

\newcommand{\Z}{{\mathbb Z}}

\newcommand{\R}{{\mathbb R}}
\newcommand{\F}{{\mathbb F}}
\newcommand{\uH}{{\mathbb H}}

\newtheorem{thm}{Theorem}
\newtheorem{lem}{Lemma}

\newtheorem{rmk}{Remark}
\newtheorem{defn}{Definition}
\newcommand{\thmref}[1]{Theorem~\ref{#1}}

\newcommand{\lemref}[1]{Lemma~\ref{#1}}

\newcommand{\rmkref}[1]{Remark~\ref{#1}}

\begin{document}

\title{On interlacing of zeros of certain family of modular forms}

\author{Ekata Saha and N. Saradha}

\address{Ekata Saha and N. Saradha\\ \newline
School of Mathematics, Tata Institute of Fundamental Research,
Homi Bhabha Road, Navy Nagar, Mumbai, 400 005, India}
\email{ekata@math.tifr.res.in}
\email{saradha@math.tifr.res.in}


\date{\today}


\begin{abstract}
Let $k=12 m(k)+s \ge 12$ for $s\in \{0,4,6,8,10,14\}$, be an even integer
and $f$ be a normalised modular form of weight $k$ with real Fourier
coefficients, written as
$$
f=E_k+\sum_{j=1}^{m(k)}a_jE_{k-12j}\Delta^j.
$$
Under suitable conditions on $a_j$
(rectifying an earlier result of Getz), we show that all the zeros
of $f$, in the standard fundamental domain for the action of ${\bf SL}(2,\Z)$
on the upper half plane, lies on the arc 
$A:= \left\{ e^{i \theta} : \frac{\pi}{2} \le \theta \le \frac{2\pi}{3} \right\}$.
Further, extending a result of Nozaki, we show that for certain family
$\{f_k\}_k$ of normalised modular forms, the zeros of $f_k$ and $f_{k+12}$
interlace on
$A^\circ:= \left\{ e^{i \theta} : \frac{\pi}{2} < \theta < \frac{2\pi}{3} \right\}$.
\end{abstract}

\maketitle

\section{\bf Introduction}

Let $\uH$ denote the complex upper half plane. Then the full modular
group ${\bf SL}(2,\Z)$ acts on $\uH$ by the transformation law
$$
z\mapsto \frac {az+b}{cz+d}
$$
for $\left( {a \atop c} {b \atop d} \right) \in {\bf SL}(2,\Z)$.
The standard fundamental domain for this action
of ${\bf SL}(2,\Z)$ on $\uH$ is the following subset of $\uH$,
$$
\F:= \left\{ |z| \ge 1, -\frac{1}{2} \le \Re(z) \le 0 \right\}
\cup \left\{ |z| > 1, 0 < \Re(z) < \frac{1}{2} \right\}.
$$
Throughout this article we take $k\ge 4$ to be an even integer.
For $z \in \uH$, the Eisenstein
series of weight $k$ for the full modular group ${\bf SL}(2,\Z)$
is defined by the following absolutely convergent series,
$$
E_k(z):=\frac{1}{2}\sum_{c,d\in\Z \atop (c,d)=1} \frac{1}{(cz+d)^k}.
$$
The Eisenstein series of weight $0$ is defined by $E_0:=1$. 
The Eisenstein series $E_k$ is a modular form of weight $k$ for
${\bf SL}(2,\Z)$. It is classical that the space of modular forms of
weight $k$ is generated by the Eisenstein series $E_k$ and cusp forms
of weight $k$. We write $k=12 m(k) +s$
with $s\in \{0,4,6,8,10,14\}$. We will use this notation for $k$ throughout
the article without any further mention. The unique normalised cusp form of
weight $12$, denoted by $\Delta$, is defined as follows:
$$
\Delta := \frac{E_4^3 \ - \ E_6^2}{1728}.
$$

Rankin and Swinnerton-Dyer \cite{RSD} proved that
for $k \ge 4$, all the zeros of the Eisenstein series
$E_k$ lie in the arc
$$
A:=\left\{ |z| = 1, -\frac{1}{2} \le \Re(z) \le 0 \right\}
= \left\{ e^{i \theta} : \frac{\pi}{2} \le \theta \le \frac{2\pi}{3} \right\}.
$$
In 2004, extending the arguments of Rankin and Swinnerton-Dyer,
Getz \cite{JG} gave a criterion for a normalised modular form of weight $k$
for ${\bf SL}(2,\Z)$, written as $f=E_k+\sum_{j=1}^{m(k)}a_j E_{k-12j} \Delta^j$, 
to have all its zeros on the arc $A$, in terms of $a_j$'s.
However, there is a rectifiable error in his proof. While estimating
$H(\theta)$ in \cite[p. 2225, eq. (2.5)]{JG}, he used
an upper bound for $R_{k-12j}, 1\le j\le m(k)$ from \cite[p. 2224, eq. (2.3)]{JG}
which is valid if $k-12j \ge 12$. But $k-12m(k)$ is always
less than $12$, unless it is $14$. We present below a corrected
version of his theorem.

Let us define
\begin{equation}\label{delta}
\delta_t=
\begin{cases}
-2 & \text{if } t=0,\\
1.009 & \text{if } t=4,\\
0.304 & \text{if } t=6,\\
0.122 & \text{if } t=8,\\
0.051 & \text{if } t=10,\\
0.022  &\text{if } t\ge 12.
\end{cases}  
\end{equation}

\begin{thm}\label{Zeros}
Let $k \ge 12$ and $f$ be a normalised modular form
of weight $k$, written as
$$
f=E_k+\sum_{j=1}^{m(k)}a_jE_{k-12j}\Delta^j, 
$$
with $a_j\in\R$ for $1\le j\le m(k)$.
Let $\epsilon := \sup_{z \in A} |\Delta(z)|$. Suppose that
\begin{equation}\label{cond1}
({3+\delta_{12}})\sum_{j=1}^{m(k)-1} |a_j|\epsilon^j 
+ (3+\delta_s)|a_{m(k)}| \epsilon^{m(k)}
\le {1-\delta_{12}}.
\end{equation}
Then $f$ has $m(k)$ zeros (other than possible zeros at $i,\rho:=e^{2\pi i/3}$)
in the fundamental domain $\F$ and they all lie on the arc $A$.
\end{thm}

\begin{rmk}\rm
Note that $\delta_{12}$ above is smaller than the $\delta$ of \cite[Theorem 1]{JG}.
This better value is due to a more accurate estimation
of a finite sum using computation. See \S2, \lemref{Eisenstein}.
Getz \cite{JG} computed $\epsilon \sim 0.004809\ldots$.
\end{rmk}

Apart from this kind of normalised modular forms there are other examples
of families of modular forms, which have been shown to have their zeros
on the arc $A$ (see \cite{RAR,AKN,DJ}). Now for the zeros of these families
of modular forms, one interesting question is to ask about their possible
{\it interlacing} property. 

\begin{defn}
Let $\alpha<\beta$. Suppose that $f,g$ are two complex valued functions
with simple zeros in the open interval $(\alpha,\beta)$. Let 
$\alpha<t_1<\cdots<t_m<\beta$ and $\alpha<t_1^*<\cdots<t_{m+1}^*<\beta$
be the zeros of $f$ and $g$, respectively. We say that zeros of $f$ and $g$ 
interlace in $(\alpha,\beta)$ if 
$$
t_j^*<t_j<t_{j+1}^* \ \ \text{for} \ \ 1\le j\le m.
$$
\end{defn}

For example, the zeros of $\cos(n \theta)$ and $\cos((n+1)\theta)$
interlace in $(0,\pi)$ for an integer $n \ge 1$.
Rankin and Swinnerton-Dyer \cite{RSD} proved that the
Eisenstein series $E_k$ has $m(k)$ simple zeros in the open arc
$A^\circ:= \left\{ e^{i \theta} : \frac{\pi}{2} < \theta < \frac{2\pi}{3} \right\}.$
For this, they considered the function
\begin{equation}\label{F_k-defn}
F_k(\theta) := e^{ik\theta/2}E_k(e^{i\theta}).
\end{equation}
This is a real valued function for $\theta$ real and it has $m(k)$
zeros in the interval $(\pi/2,2\pi/3)$ and so does $E_k(e^{i\theta})$
in $A^\circ$. Note that $m(k+12)=m(k)+1$. Hence one may look
for the interlacing property for the zeros of $F_k(\theta)$ and
$F_{k+12}(\theta)$ for $\theta \in (\pi/2,2\pi/3)$. In this instance,
we say that the zeros of $E_k(e^{i\theta})$ and
$E_{k+12}(e^{i\theta})$ interlace in $A^\circ$.
This was predicted by Gekeler \cite{EUG} and proved by Nozaki~\cite{HN}.

For the zeros of certain families of weakly holomorphic modular forms
considered by Asai, Kaneko and Ninomiya \cite{AKN}, their interlacing property was
established by Jermann
\cite{JJ}. Similar properties for the zeros of the weakly holomorphic modular forms,
studied by Duke and Jenkins \cite{DJ}, were proved by Jenkins and Pratt \cite{JP}.
Here we establish the interlacing
property of the zeros of certain family of normalised modular forms
that were considered in \thmref{Zeros}.

\begin{thm}\label{interlacing}
For each $k \ge 12$, let
$(a_j^{(k)})_{1\le j \le m(k)}$ be real numbers such that
\begin{equation}\label{cond2}
(3+\delta_{12}) \sum_{j=1}^{m(k)-1} |a_j^{(k)}| \epsilon^j 
+ (3+\delta_s) |a^{(k)}_{m(k)}| \epsilon^{m(k)}
\le 20 \left( \frac{1}{2} \right)^{k/2},
\end{equation}
where $\delta_s$ is as in \eqref{delta}. 
Then for the family of normalised modular forms $(f_k)_k$ for
${\bf SL}(2,\Z)$ defined by
$$
f_k:= E_k+\sum_{j=1}^{m(k)}a_j^{(k)} E_{k-12j}\Delta^j,
$$
the zeros of $f_k$ in the fundamental domain $\F$ lie on the arc $A$.
Further, the zeros of $f_k$ and those of $f_{k+12}$ interlace in $A^\circ$
for each $k \ge 12$.
\end{thm}

\begin{rmk}\rm
By \eqref{cond2}, we see that \eqref{cond1} is satisfied. Hence by \thmref{Zeros},
all the zeros of $(f_k)_k$ lie on the arc $A$, thus giving the first assertion of
\thmref{interlacing}.
\end{rmk}

\begin{rmk}\rm
Nozaki's result is a special case of \thmref{interlacing} when $a_j^{(k)}=0$
for all $1\le j \le m(k)$.
\end{rmk}

For proving the interlacing of the zeros of the Eisenstein series $E_k(z)$,
Nozaki showed that $F_k(\theta)$ as defined in \eqref{F_k-defn}, 
is very well approximated by $2 \cos(k\theta/2)$ for $\theta \in (\pi/2,2\pi/3)$.
This is an important step in his method. We are able to show that
\begin{equation}\label{G_k-defn}
G_k(\theta) := e^{ik\theta/2} f_k(e^{i\theta})
\end{equation}
is also well approximated by $2 \cos(k\theta/2)$ for $\theta \in (\pi/2,2\pi/3)$.
See \S 3.1 for details. Both these functions have $m(k)$ zeros in $(\pi/2,2\pi/3)$.
If $\alpha$ is a zero of $\cos(k\theta/2)$ for $\theta \in (\pi/2,2\pi/3)$, then there
is a neighbourhood of $\alpha$, say $(\alpha-\epsilon,\alpha+\epsilon)$,
containing exactly one zero $\alpha^*$ of $G_k(\theta)$. It can be easily seen that,
the zeros of $\cos(k \theta/2)$ and $\cos((k+12)\theta/2)$
interlace in $(\pi/2,2\pi/3)$ (see \S 3.2). Thus there exist successive zeros
$\beta,\gamma$ of $\cos((k+12)\theta/2)$ with
$$
\beta< \alpha < \gamma.
$$
Again, 
there exist intervals of the form
$(\beta-\delta,\beta+\delta)$ and $(\gamma-\mu,\gamma+\mu)$,
each containing exactly one zero of $\cos((k+12)\theta/2)$,
say $\beta^*,\gamma^*$ respectively. Thus if
$$
\beta+\delta < \alpha-\epsilon < \alpha+\epsilon < \gamma-\mu,
$$
then we obtain that 
$$
\beta^*< \alpha^* < \gamma^*.
$$
This argument is used to show that
the zeros of $G_k (\theta)$ and $G_{k+12}(\theta)$
interlace in $(\pi/2,23\pi/36)$ (see \S 3.2, 4.1). 
This method does not work as we approach $2\pi/3$. For proving
the interlacing property in the remaining interval we consider
the interval $(19\pi/32,2\pi/3)$, which overlaps with $(\pi/2,23\pi/36)$.
Here the method depends on analysing different cases
according to the increasing or decreasing property
of the cosine function at their respective zeros (see \S 4.2).

\section{\bf A lemma and Proof of \thmref{Zeros}}

We begin this section with the following lemma. This will be used in the proof
of both the Theorems \ref{Zeros} and \ref{interlacing}.

\begin{lem}\label{Eisenstein}
Let $k\ge 4$ and
$F_k(\theta):= e^{ik\theta/2}E_k(e^{i\theta})$ for $\theta\in [\pi/2,2\pi/3]$.
Then
$$
F_k(\theta)=2\cos(k\theta/2)+\left(\frac{1}{2\cos(\theta/2)}\right)^k
+ \left(\frac{1}{2i\sin(\theta/2)}\right)^k + P_k(\theta), 
$$
where 
$$
|P_k(\theta)|<
\begin{cases}
0.759 & \text{if } k=4,\\
0.179 & \text{if } k=6,\\
0.059 & \text{if } k=8,\\
0.019 & \text{if } k=10,\\
0.359(\frac{1}{2})^{k/2}  & \text{if } k\ge 12.
\end{cases}  
$$
\end{lem}

\begin{rmk}\label{bound}\rm
From \lemref{Eisenstein} we obtain that
$$
\sup_{\theta \in  [\pi/2,2\pi/3]} |F_k(\theta)|=\sup_{z \in A} |E_k(z)|<
\begin{cases}
4.009 & \text{if } k=4,\\
3.304 & \text{if } k=6,\\
3.122 & \text{if } k=8,\\
3.051 & \text{if } k=10,\\
3 + 1.359 (\frac{1}{2})^{k/2}  &\text{if } k\ge 12.
\end{cases}  
$$
In particular, we shall use the following bounds for the proof of
\thmref{Zeros}. For $k \ge 4$,
$$
\sup_{\theta \in  [\pi/2,2\pi/3]} |F_k(\theta)|
\le 3 +  \left|\frac{1}{2i\sin(\theta/2)}\right|^k + |P_k(\theta)|
< 3 + \delta_k,
$$
where $\delta_k$ is as in \eqref{delta}.
\end{rmk}

\subsection{Proof of \lemref{Eisenstein}}
For $N \ge 1$, let us define
$$
\sigma_N(\theta):=\frac{1}{2}\sum_{c,d\in\Z \atop {c^2+d^2=N \atop (c,d)=1}} 
\frac{1}{(ce^{i\theta/2}+de^{-i\theta/2})^k}.
$$
In the above definition, whenever empty sum appears, it is
assumed to be $0$. Now 
$$
F_k(\theta) =\frac{1}{2}\sum_{c,d\in\Z \atop (c,d)=1} 
\frac{1}{(ce^{i\theta/2}+de^{-i\theta/2})^k},
$$
for $\theta \in [\pi/2, 2\pi/3]$. 
Since the series defining the Eisenstein series
of weight $k \ge 4$ is absolutely convergent, we can write
$$
F_k(\theta)= \sum_{N \ge 1} \sigma_N(\theta).
$$
One can easily see that
$$
\sigma_1(\theta)=2\cos(k\theta/2)
\ \ \text{and} \ \ 
\sigma_2(\theta)=
\left(\frac{1}{2\cos(\theta/2)}\right)^k
+ \left(\frac{1}{2i\sin(\theta/2)}\right)^k.
$$
Hence
\begin{equation}\label{F_k}
F_k(\theta)=2\cos(k\theta/2)+\left(\frac{1}{2\cos(\theta/2)}\right)^k
+ \left(\frac{1}{2i\sin(\theta/2)}\right)^k + P_k(\theta), 
\end{equation}
where
$$
P_k(\theta):= \sum_{N \ge 5} \sigma_N(\theta).
$$
Now we split the above sum and write
\begin{equation}\label{P_k}
P_k(\theta)= \sum_{5 \le N \le A} \sigma_N(\theta)+ \sum_{N \ge B} \sigma_N(\theta),
\end{equation}
for some integer $A \ge 5$ and $B$, the least integer
larger than $A$ which can be written as sum of squares of
two co-prime integers.

Note that for any two real numbers $c,d$ one has
$|cd| \le \frac{c^2+d^2}{2}$. Since $-1\le 2\cos\theta \le 0$ for 
$\theta\in [\pi/2,2\pi/3]$,
we obtain
\begin{equation}\label{mod}
|ce^{i\theta/2}+de^{-i\theta/2}|^2
=c^2+d^2+2cd\cos \theta \ge c^2+d^2 - |cd| \ge \frac{c^2+d^2}{2}.
\end{equation}
Thus we get
$$
\left| \sum_{5 \le N \le A} \sigma_N(\theta) \right|
\le \frac{1}{2} \sum_{5 \le N \le A} 
\sum_{c,d\in\Z \atop {c^2+d^2=N \atop (c,d)=1}} 
\frac{1}{(c^2 + d^2 - |cd|)^{k/2}}.
$$
For any natural number
$N$, there are at most $2(2N^{1/2}+1)$ pairs $(c,d)$ of integers
such that $c^2+d^2=N$. Hence for the second sum in \eqref{P_k},
using \eqref{mod} we get
\begin{align*}
\left| \sum_{N \ge B} \sigma_N(\theta) \right|
& \le \sum_{N \ge B} \left( \frac{2}{N} \right)^{k/2} (2N^{1/2}+1)\\
& \le \left(2+ \frac{1}{\sqrt{B}} \right) 2^{k/2} \sum_{N \ge B} N^{(1-k)/2} \\
& \le \left(2+ \frac{1}{\sqrt{B}} \right) 2^{k/2}
\int_{B-1}^{\infty} x^{(1-k)/2}dx\\
& = \left(2+ \frac{1}{\sqrt{B}} \right) \frac{2^{(k+2)/2}}{k-3}
\left( \frac{1}{B-1} \right)^{(k-3)/2}.
\end{align*}
So we have
\begin{equation}\label{P_k-bound}
|P_k(\theta)| \le \frac{1}{2} \sum_{5 \le N \le A} 
\sum_{c,d\in\Z \atop {c^2+d^2=N \atop (c,d)=1}} 
\frac{1}{(c^2 + d^2 - |cd|)^{k/2}} +
\left(2+ \frac{1}{\sqrt{B}} \right) \frac{2^{(k+2)/2}}{k-3}
\left( \frac{1}{B-1} \right)^{(k-3)/2}.
\end{equation}
Let us denote the right hand side of \eqref{P_k-bound} by $p_k$.
Using a C++ programming we obtain optimal values of $p_k$ as given
in Table 1. The values of $p_k$ is best possible up to second
decimal place. Our choice of $A$ is also indicated in the table.
\begin{center}
\begin{tabular}{ | l | l | l |}
\hline
$k$ & $A$ & $p_k$ \\ \hline \hline
$4$ & $3 \times 10^6$ & $0.759$  \\ \hline
$6$ & $536$ & $0.179$ \\ \hline
$8$ & $23$ & $0.059$  \\ \hline
$10$ & $14$ & $0.019$\\ \hline
$12$ or more & $20$ & $0.359 \left(\frac{1}{2} \right)^{k/2}$ \\ \hline
\end{tabular}
\\ \medskip Table 1
\end{center}
The computation of the values of $p_k$, when $k \ge 6$ has taken only a few
seconds, whereas for $p_4$, the programme ran for about two hours.
Our proof is now complete. \qed

\subsection{Proof of \thmref{Zeros}}
It is easy to see that $F_k(\theta)$ is real valued for
$\theta \in [\pi/2,2\pi/3]$. Further,
$$
\Delta(e^{i\theta})e^{6i\theta}=\frac{F_4^3(\theta)-F_6^2(\theta)}{1728}\in\R \ \  \text{for}\ \  \theta\in [\pi/2,2\pi/3].
$$
Now, 
\begin{align*}
G(\theta):=e^{ik\theta/2}f(e^{i \theta})
&= e^{ik\theta/2}E_k(e^{i \theta})+
\sum_{j=1}^{m(k)}a_j e^{i(k-12j)\theta/2} E_{k-12j}(e^{i \theta})
\Delta(e^{i \theta})^j e^{6ij\theta}\\
&= F_k(\theta) + \sum_{j=1}^{m(k)}a_j F_{k-12j}(\theta)
\Delta(e^{i \theta})^j e^{6ij\theta}.
\end{align*}
Since $a_j \in \R$ for $1 \le j \le m(k)$,
we see that $G(\theta)$ is real valued for $\theta \in [\pi/2,2\pi/3]$. 
Also by \lemref{Eisenstein} we have,
$$
G(\theta)
= 2\cos(k\theta/2)+\left(\frac{1}{2\cos(\theta/2)}\right)^k
+ \left(\frac{1}{2i\sin(\theta/2)}\right)^k + P_k(\theta)\\
+ \sum_{j=1}^{m(k)}a_j F_{k-12j}(\theta)
\Delta(e^{i \theta})^j e^{6ij\theta}.
$$
We show that
\begin{equation}\label{to-show}
|G(\theta) - 2 \cos(k\theta/2)|< 2.
\end{equation}
By \rmkref{bound},
\begin{align*}
|G(\theta) - 2 \cos(k\theta/2)| 
& < 1 + \delta_{12} + \sum_{j=1}^{m(k)} |a_j| |F_{k-12j}(\theta)|\epsilon^j\\
& \le 1 + \delta_{12} + (3+\delta_{12}) \sum_{j=1}^{m(k)-1} |a_j| \epsilon^j
+ (3+\delta_s) |a_{m(k)}| \epsilon^{m(k)}.
\end{align*}
Thus, $|G(\theta) - 2 \cos(k\theta/2)|< 2$  by our hypothesis.
Now we argue as in \cite{RSD}. By \eqref{to-show}, we get that
between two consecutive extremum points of $\cos(k\theta/2)$,
there is a zero of $G(\theta)$. Now the extremum points of $\cos(k\theta/2)$
for $\theta \in [\pi/2,2\pi/3]$ are given by $\frac{2 \pi n}{k}$,
where $\frac{k}{4}\le n \le \frac{k}{3}$ i.e.
$3m(k)+\frac{s}{4}\le n \le 4m(k)+\frac{s}{3}$. So we have
$m(k)+1$ such $n$'s. Therefore $G(\theta)$ has $m(k)$ zeros on
$(\pi/2,2\pi/3)$. This, together with the valence formula, completes the proof.
\qed

\section{\bf Properties of $G_k(\theta)$}

\subsection{Approximation of $G_k(\theta)$ by $2\cos(k \theta/2)$}
We write
$G_k(\theta)=e^{ik\theta/2}f_k(e^{i\theta})$ as
$$
G_k(\theta) = 2 \cos(k\theta/2) + S_k(\theta),
$$
where by \lemref{Eisenstein},
$$
S_k(\theta) = \left(\frac{1}{2\cos(\theta/2)}\right)^k
+ \left(\frac{1}{2i\sin(\theta/2)}\right)^k + P_k(\theta)
+ \sum_{j=1}^{m(k)}a_j^{(k)} F_{k-12j}(\theta) \Delta(e^{i \theta})^j e^{6ij\theta},
$$
with $|P_k(\theta)|< 0.359 \left(\frac{1}{2}\right)^{k/2}$. Let
\begin{equation}\label{Q_k}
Q_k(\theta) := \left(\frac{1}{2i\sin(\theta/2)}\right)^k + P_k(\theta)
+ \sum_{j=1}^{m(k)}a_j^{(k)} F_{k-12j}(\theta) \Delta(e^{i \theta})^j e^{6ij\theta}.
\end{equation}
Then by \lemref{Eisenstein}, \rmkref{bound} and our hypothesis we get
\begin{equation}\label{our-bound}
|Q_k(\theta)|
\le \left| \left(\frac{1}{2i\sin(\theta/2)}\right)^k\right| + |P_k(\theta)|
+ \sum_{j=1}^{m(k)}|a_j^{(k)}| |F_{k-12j}(\theta)| \epsilon^j
< 21.359 \left(\frac{1}{2}\right)^{k/2}.
\end{equation}
Thus
$$
G_k(\theta) = 2 \cos(k\theta/2)+ S_k(\theta) = 
2 \cos(k\theta/2)+ \left(\frac{1}{2\cos(\theta/2)}\right)^k
+ Q_k(\theta)
$$
with $|Q_k(\theta)| < 21.359 \left(\frac{1}{2}\right)^{k/2}$.
By following the proof of \cite[Lemma 4.1]{HN}, we observe that
\begin{equation}\label{S_k-bound}
0<S_k(\theta)<1, \ \ \text{whenever} \ k\ge 24 \ \text{and} \  \theta\in [19\pi/32,2\pi/3-2\pi/3k].
\end{equation}
Further, we argue as in the proof of \cite[Lemma 4.4]{HN}
to obtain that for $\theta \in [19\pi/32, \alpha_{m(k)}]$, the function
$$
g(\theta):=(2\cos(\theta /2))^{-k}-(2\cos(\theta /2))^{-(k+12)}
$$
is minimised at $\theta=19\pi/32$. Hence for $k \ge 24$
\begin{align*}
S_k(\theta) - S_{k+12}(\theta)
& = g(\theta) + Q_k(\theta) - Q_{k+12}(\theta)\\
& > g\left(\frac{19\pi}{32}\right) - \frac{21.359}{2^{k/2}} - \frac{21.359}{2^{(k+12)/2}}\\
& > \left(2\cos\left(\frac{19\pi}{64}\right)\right)^{-k}
\left(1-\left(2\cos\left(\frac{19\pi}{64}\right)\right)^{-12} \right) - \frac{24.919}{2^{k/2}}\\
& > \frac{0.877}{(1.192)^k} - \frac{24.919}{(1.414)^{k}}>0.
\end{align*}
Thus for $k \ge 24$ and $\theta \in [19\pi/32, \alpha_{m(k)}]$,
\begin{equation}\label{S_k-ineq}
S_k(\theta) > S_{k+12}(\theta).
\end{equation}

\subsection{Zeros of $G_k(\theta)$ and $G_{k+12}(\theta)$}
Let $\alpha_1< \cdots < \alpha_{m(k)}$ and $\beta_1 < \cdots < \beta_{m(k)+1}$
denote the zeros of $\cos(k\theta/2)$ and $\cos((k+12)\theta/2)$ in $(\pi/2,2\pi/3)$,
respectively. Then
\begin{equation}\label{alpha_j}
\alpha_j=
\begin{cases}
(\frac{1}{2} + \frac{2j-1}{k})\pi & \ \text{if} \ k \equiv 0 \bmod 4,\\
(\frac{1}{2} + \frac{2j}{k})\pi & \ \text{if} \ k \equiv 2 \bmod 4,
\end{cases}
\end{equation}
for $1 \le j \le m(k)$ and

\begin{equation}\label{beta_j}
\beta_j=
\begin{cases}
(\frac{1}{2} + \frac{2j-1}{k+12})\pi & \ \text{if} \ k \equiv 0 \bmod 4,\\
(\frac{1}{2} + \frac{2j}{k+12})\pi & \ \text{if} \ k \equiv 2 \bmod 4,
\end{cases}
\end{equation}
for $1 \le j \le m(k)+1$. We observe the following properties: for
$1 \le j \le m(k)$,
\begin{equation}\label{int-cos}
\beta_j < \alpha_j < \beta_{j+1} 
\end{equation}
\begin{equation}\label{distance}
\min_{1\le j \le m(k)} \{\alpha_j-\beta_j,\beta_{j+1} - \alpha_j\}
\ge \frac{12 \pi}{k(k+12)}
\end{equation}
and
\begin{equation}\label{alpha-distance}
\alpha_{j+1} - \alpha_j = \frac{2 \pi}{k}\ \ \text{ for } \ 1 \le j \le m(k)-1.
\end{equation}
Hence we have,
\begin{equation}\label{cos-product-1}
\cos \left(\frac{k}{2} \left( \alpha_j - \frac{6 \pi}{k(k+12)} \right) \right)
\cos \left(\frac{k}{2} \left( \alpha_j + \frac{6 \pi}{k(k+12)} \right) \right)<0
\end{equation}
and
\begin{equation}\label{cos-product-2}
\cos \left(\frac{k}{2} \left( \alpha_j - \frac{\pi}{3k}\right) \right)
\cos \left(\frac{k}{2} \left( \alpha_j + \frac{\pi}{3k}\right) \right)<0.
\end{equation}

\begin{rmk}\label{k-cond} \rm
From \eqref{alpha_j} and \eqref{beta_j} it follows that $k \ge 24$,
whenever $\alpha_j+ \frac{6 \pi}{k(k+12)} > \frac{23\pi}{36}$
or $\alpha_j \ge \frac{19\pi}{32}+ \frac{\pi}{3k}$
or $\beta_j+ \frac{6 \pi}{(k+12)(k+24)} > \frac{23\pi}{36}$.
\end{rmk}

We now associate the zeros of $G_{k}(\theta)$ with the zeros of $\cos(k\theta/2)$.

\begin{lem}\label{l1}
Let $k \ge 12$ and $\alpha_1^*< \cdots < \alpha_{m(k)}^*$
be the zeros of $G_{k}(\theta)$ in $(\pi/2, 2\pi/3)$. Then
$$
\alpha_j^* \in \left( \alpha_j - \frac{6 \pi}{k(k+12)},
\alpha_j + \frac{6 \pi}{k(k+12)} \right)
$$
for all $j$ such that $\alpha_j+ \frac{6 \pi}{k(k+12)} \le \frac{23\pi}{36}$ and
$$
\alpha_j^* \in \left( \alpha_j - \frac{\pi}{3k},
\alpha_j + \frac{\pi}{3k} \right)
$$
for all $j$ such that $\alpha_j \ge 19\pi/32 + \pi/3k$. 
In particular, $\alpha_j^* \in \left( \alpha_j - \frac{\pi}{3k},
\alpha_j + \frac{\pi}{3k} \right)$ for $1\le j \le m(k)$.
\end{lem}

\begin{proof}[\bf Proof of \lemref{l1}]
Observe that
$$
\left| 2 \cos \left(\frac{k}{2} \left( \alpha_j - \frac{6 \pi}{k(k+12)}
\right) \right) \right|
=\left| 2 \cos \left(\frac{k}{2} \left( \alpha_j + \frac{6 \pi}{k(k+12)}
\right) \right) \right|
=2\sin\left( \frac{3 \pi}{k+12}\right).
$$
We first claim that $2\sin\left( \frac{3 \pi}{k+12}\right) > |S_k(\theta)|$
for $\theta \in (\pi/2, 23\pi/36]$. Note that
$$
2\sin\left( \frac{3 \pi}{k+12}\right)
> 2\left( \frac{3 \pi}{k+12} - \frac{1}{6} \left( \frac{3 \pi}{k+12}\right)^3\right)
> \frac{18.365}{k+12}.
$$
Hence it is enough to show that
\begin{align*}
\frac{18.365}{k+12} - |S_k(\theta)|
&= \frac{18.365}{k+12} - \left(\frac{1}{2\cos(\theta/2)}\right)^k - |Q_k(\theta)|\\
&> \frac{18.365}{k+12} - \left(\frac{1}{2\cos(\theta/2)}\right)^k -
\frac{21.359}{2^{k/2}}>0.
\end{align*}
Now this is true if
$$
2^{k/2} \left(18.365 (2\cos(\theta/2))^k - (k+12) -
\frac{21.359(k+12)}{2^{k/2}} (2\cos(\theta/2))^k \right) >0,
$$
In particular if
$$
10.355 (2\cos(\theta/2))^k - (k+12) >0.
$$
The quantity $\sqrt[k]{\frac{k+12}{10.355}}$ is maximum for $k=12$,
and therefore $\sqrt[k]{\frac{k+12}{10.355}} < 1.073$.
Note that for $\theta \in (\pi/2, 23\pi/36]$, $2\cos(\theta/2) > 1.074$.
This proves our first claim.

Since
$$
\left| 2 \cos \left(\frac{k}{2} \left( \alpha_j \pm \frac{6 \pi}{k(k+12)}
\right) \right) \right| > |S_k(\theta)|
$$
for $\theta \in (\pi/2, 23\pi/36]$, we get that
$$
G_k\left( \alpha_j \pm \frac{6 \pi}{k(k+12)} \right)
= 2 \cos \left(\frac{k}{2} \left( \alpha_j \pm \frac{6 \pi}{k(k+12)}
\right) \right) + S_k\left( \alpha_j \pm \frac{6 \pi}{k(k+12)} \right)
$$
have same sign as $2 \cos \left(\frac{k}{2} \left( \alpha_j \pm \frac{6 \pi}{k(k+12)}
\right) \right)$. Then from \eqref{cos-product-1} we get that
$$
G_k\left( \alpha_j - \frac{6 \pi}{k(k+12)} \right)
G_k\left( \alpha_j + \frac{6 \pi}{k(k+12)} \right)<0.
$$
This implies that there exists a zero of $G_{k}(\theta)$ in the interval
$\left( \alpha_j - \frac{6 \pi}{k(k+12)},\alpha_j + \frac{6 \pi}{k(k+12)} \right)$
for all $j$ such that $\alpha_j+ \frac{6 \pi}{k(k+12)} \le \frac{23\pi}{36}$.

It is easy to see from \eqref{alpha_j} that $\alpha_{m(k)} \le \frac{2\pi}{3}-\frac{\pi}{k}$.
Hence we get $\alpha_{m(k)}+\frac{\pi}{3k} \le \frac{2\pi}{3}-\frac{2\pi}{3k}$.
Since $\cos(k \alpha_j/2)=0$, we have
$$
\left| 2 \cos \left(\frac{k}{2} \left( \alpha_j - \frac{\pi}{3k}
\right) \right) \right|
=\left| 2 \cos \left(\frac{k}{2} \left( \alpha_j + \frac{\pi}{3k}
\right) \right) \right|=1.
$$
As $G_k(\theta) = 2 \cos(k\theta/2) + S_k(\theta)$, by 
\eqref{S_k-bound} we get that 
the sign of $G_k\left( \alpha_j \pm \frac{\pi}{3k}\right)$
is same as that of
$2 \cos \left(\frac{k}{2} \left( \alpha_j \pm \frac{\pi}{3k} \right) \right)$,
whenever $\alpha_{j} \ge \frac{19\pi}{32} + \frac{\pi}{3k}$.
Hence it follows from \eqref{cos-product-2} that
$$
G_k \left( \alpha_j - \frac{\pi}{3k}\right)
G_k \left( \alpha_j + \frac{\pi}{3k}\right)<0.
$$
This implies that there exists a zero of $G_{k}(\theta)$ in the interval
$\left( \alpha_j - \frac{\pi}{3k},\alpha_j + \frac{\pi}{3k} \right)$
for all $j$ such that $\alpha_j \ge \frac{19\pi}{32} + \frac{\pi}{3k}$.

From \rmkref{k-cond} we deduce that a zero $\alpha_j$ of
$\cos(k\theta/2)$ in $(\pi/2,2\pi/3)$ satisfies
$\alpha_j+ \frac{6 \pi}{k(k+12)} \le \frac{23\pi}{36}$
or $\alpha_j \ge \frac{19\pi}{32} + \frac{\pi}{3k}$. Moreover,
$$
\left( \alpha_j - \frac{6 \pi}{k(k+12)},\alpha_j + \frac{6 \pi}{k(k+12)} \right)
\subset \left( \alpha_j - \frac{\pi}{3k},\alpha_j + \frac{\pi}{3k} \right).
$$
By \eqref{alpha-distance} all the intervals of the form
$\left( \alpha_j - \frac{\pi}{3k},\alpha_j + \frac{\pi}{3k} \right)$
are disjoint for $1 \le j \le m(k)$. We have shown above that
each of them contains at least one zero of $G_{k}(\theta)$. We also know that
$G_{k}(\theta)$ has exactly $m(k)$ zeros in $(\pi/2,2\pi/3)$. Since
the zeros of $\cos(k\theta/2)$ and $G_{k}(\theta)$ have been labelled 
as per the increasing order of magnitude, the assertion of the lemma follows.
\end{proof}

\section{\bf Proof of \thmref{interlacing}}

We closely follow the arguments of Nozaki \cite{HN}.
However at  several places our arguments are simpler.
We know that a zero $\alpha_j$ of
$\cos(k\theta/2)$ in $(\pi/2,2\pi/3)$ satisfies
$\alpha_j+ \frac{6 \pi}{k(k+12)} \le \frac{23\pi}{36}$
or $\alpha_j \ge \frac{19\pi}{32} + \frac{\pi}{3k}$. Hence we
prove \thmref{interlacing} for these two cases.

\subsection{Case I}
Let $\alpha_j \in \left( \frac{\pi}{2}, \frac{23\pi}{36}-\frac{6 \pi}{k(k+12)} \right]$.
Then we prove the following:
\begin{itemize}
\item[(i)] $\beta_j^* < \alpha_j^*$.
\item[(ii)] $\alpha_j^* < \beta_{j+1}^*$ if
$\beta_{j+1}+ \frac{6 \pi}{(k+12)(k+24)} \le \frac{23\pi}{36}$.
\end{itemize}

\subsubsection{\bf Proof of (i)}
Since $\alpha_j+ \frac{6 \pi}{k(k+12)} \le \frac{23\pi}{36}$,
we have $\beta_j+ \frac{6 \pi}{(k+12)(k+24)} \le \frac{23\pi}{36}$.
Applying \lemref{l1} for $G_k(\theta)$ and $G_{k+12}(\theta)$, we get
$$
\alpha_j^* \in \left( \alpha_j - \frac{6 \pi}{k(k+12)},
\alpha_j + \frac{6 \pi}{k(k+12)} \right) \ \text{and} \
\beta_j^* \in \left( \beta_j - \frac{6 \pi}{(k+12)(k+24)},
\beta_j + \frac{6 \pi}{(k+12)(k+24)} \right).
$$

Therefore by \eqref{distance}, we have
$$
\beta_j^*< \beta_j + \frac{6 \pi}{(k+12)(k+24)}
< \beta_j + \frac{6 \pi}{k(k+12)}
\le \alpha_j - \frac{6 \pi}{k(k+12)} < \alpha_j^*.
$$
\subsubsection{\bf Proof of (ii)}
When $\beta_{j+1}+ \frac{6 \pi}{(k+12)(k+24)} \le \frac{23\pi}{36}$, by \lemref{l1}
$$
\beta_{j+1}^* \in \left( \beta_{j+1} - \frac{6 \pi}{(k+12)(k+24)},
\beta_{j+1} + \frac{6 \pi}{(k+12)(k+24)} \right).
$$
Therefore by \eqref{distance}, we have
$$
\alpha_j^*< \alpha_j + \frac{6 \pi}{k(k+12)}
\le \beta_{j+1} - \frac{6 \pi}{k(k+12)}
< \beta_{j+1} - \frac{6 \pi}{(k+12)(k+24)} < \beta_{j+1}^*.
$$

\subsection{Case II} Let
$\alpha_j \in \left[ \frac{19\pi}{32} + \frac{\pi}{3k},\frac{2\pi}{3} \right)$.
Then we prove the following:
\begin{itemize}
\item[(iii)] $\beta_j^* < \alpha_j^*$ if
$\beta_{j} \ge \frac{19\pi}{32} + \frac{\pi}{3(k+12)}$.
\item[(iv)] $\alpha_j^* < \beta_{j+1}^*$.
\end{itemize}
For an integer $n$,
$$
\cos(n \pi /2)
= \begin{cases}
1 & \text{if } n \equiv 0 \bmod 4,\\
-1 & \text{if } n \equiv 2 \bmod 4,\\
0 & \text{otherwise} .
\end{cases}
$$
Further $\cos(n \theta)$ is decreasing at $\pi/2$ if $n \equiv 1 \bmod 4$ and
increasing at $\pi/2$ if $n \equiv 3 \bmod 4$. Thus if
$\alpha$ denotes the first zero of $\cos(n \theta)$
in $(\pi/2,\infty)$, then $\cos(n \theta)$ is
decreasing at $\alpha$ if $n \equiv 0,3 \bmod 4$ and
increasing at $\alpha$ if $n \equiv 1,2 \bmod 4$. Hence if
$\beta$ denotes the first zero of $\cos((n+6) \theta)$
in $(\pi/2,\infty)$, then $\cos((n+6) \theta)$ is
increasing at $\beta$ if $n \equiv 0,3 \bmod 4$ and
decreasing at $\beta$ if $n \equiv 1,2 \bmod 4$.

Therefore, for all $1 \le j \le m(k)$, if $\cos(k \theta/2)$
is increasing (resp. decreasing) at $\alpha_j$,
then $\cos((k+12) \theta/2)$
is decreasing (resp. increasing) at $\beta_j$.
We consider two subcases.

\begin{itemize}

\item[(a)] The function $\cos(k\theta/2)$ is increasing at $\alpha_j$.
\item[(b)] The function $\cos(k\theta/2)$ is decreasing at $\alpha_j$.
\end{itemize}
Note that by \rmkref{k-cond}, we have $k \ge 24$. Now for $k \ge 24$ and
$\theta \in [19\pi/32,2\pi/3-2\pi/3k]$, $S_k(\theta)>0$ (see \eqref{S_k-bound}).
Hence $\alpha_j^* \lessgtr \alpha_j$ according as
$\cos(k\theta/2)$ is increasing or decreasing at $\alpha_j$.

\subsubsection{\bf Proof of (iii) and (iv) when (a) occurs}
It follows by our earlier analysis that $\cos((k+12)\theta/2)$ is increasing at $\beta_{j+1}$.
Also $\beta_{j+1} > \alpha_j > \frac{19\pi}{32} + \frac{\pi}{3(k+12)}$. Hence
\begin{equation}\label{101}
\alpha_j^* < \alpha_j \ \ \text{and} \ \ \beta_{j+1}^*< \beta_{j+1}.
\end{equation}
Further
$$
\beta_j< \beta_j^* \ \ \text{if} \ \ \beta_{j} \ge \frac{19\pi}{32} + \frac{\pi}{3(k+12)}.
$$
The Case (a) is described pictorially below.
\begin{figure}[H]
\begin{center}
 \begin{minipage}[b]{0.4\textwidth}
  \includegraphics[width=5cm]{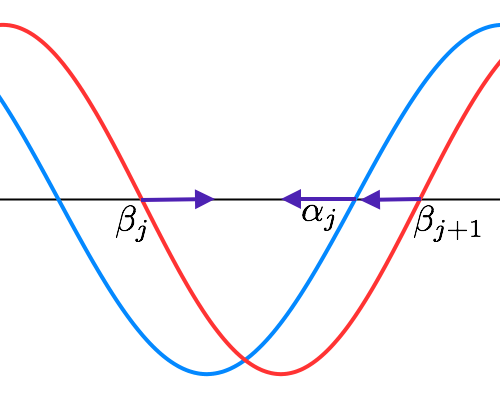}
  \label{p2}
   \end{minipage}
   \end{center}
\end{figure}
In the above one and other picture below, the respective arrows indicate
the direction in which $\beta_j^*, \alpha_j^*$ and $\beta_{j+1}^*$ lie.
We first prove (iii). In fact,
we show that the intersection point of the two cosine curves
between $\beta_j$ and $\alpha_j$ separates $\beta_j^*$ and $\alpha_j^*$.
The function
$$
2\cos((k+12)\theta/2)-2\cos(k\theta/2)=-4\sin((k+6)\theta/2) \sin(3\theta),
$$
has a zero between $\beta_j$ and $\alpha_j$, say
$\gamma_j$. For $1 \le j \le m(k)$,
\begin{equation}\label{gamma_j}
\gamma_j=
\begin{cases}
(\frac{1}{2} + \frac{2j-1}{k+6})\pi & \ \text{if} \ k \equiv 0 \bmod 4,\\
(\frac{1}{2} + \frac{2j}{k+6})\pi & \ \text{if} \ k \equiv 2 \bmod 4.
\end{cases}
\end{equation}
We prove that $\beta_j^* \in (\beta_j,\gamma_j)$ and
$\alpha_j^* \in (\gamma_j,\alpha_j)$. By \eqref{S_k-bound},
$G_k(\alpha_j)>0$ and $G_{k+12}(\beta_j)>0$. Thus, we are led
to show that $G_k(\gamma_j),G_{k+12}(\gamma_j)<0$. As in the proof of
\cite[Lemma 4.2]{HN}, it follows that
$$
|2\cos(k\gamma_j/2)|>S_k(\theta)
\ \ \text{and} \ \
|2\cos((k+12)\gamma_j/2)|>S_{k+12}(\theta),
$$
for any $j$ such that $\gamma_j \ge \frac{19\pi}{32}$
and $\theta \in [19\pi/32,2\pi/3)$.

Since $\cos(k\theta/2)$ is increasing
at $\alpha_j$ and $\cos(k\alpha_j/2)=0$, we get $\cos(k\gamma_j/2)<0$.
Hence $G_k(\gamma_j)<0$. Similarly one obtains that $G_{k+12}(\gamma_j)<0$.
This completes the proof of (iii) when (a) occurs.

Next we show (iv) i.e. $\alpha_j^* <  \beta_{j+1}^*$. Now if $\alpha_j \le \beta_{j+1} - \frac{\pi}{3(k+12)}$,
then by \eqref{101} and \lemref{l1},
$$
\alpha_j^* < \alpha_j \le \beta_{j+1} - \frac{\pi}{3(k+12)} <\beta_{j+1}^*,
$$
which proves (iv). Hence we need to consider only
$$
\beta_{j+1}-\frac{\pi}{3(k+12)} < \alpha_j  \ \ \text{and also} \ \ 
\alpha_j^*, \beta_{j+1}^* \in \left(\beta_{j+1}-\frac{\pi}{3(k+12)}, \alpha_j\right).
$$
From \eqref{S_k-ineq} we have, $S_k(\theta) > S_{k+12}(\theta)$ for
$\theta \in [\beta_{j+1}-\pi/3(k+12), \alpha_j]$. Further,
the function $\cos(k \theta /2)-\cos((k+12)\theta /2)$ takes positive value
in the interval $(\gamma_j,\gamma_{j+1})$, where $\gamma_j$'s are as in
\eqref{gamma_j}. It is easy to check that
$[\beta_{j+1}-\pi/3(k+12), \alpha_j] \subset (\gamma_j,\gamma_{j+1})$.
Therefore, we obtain that
$$
G_k(\theta) - G_{k+12}(\theta)
= 2\cos(k \theta /2)- 2\cos((k+12)\theta /2)
+ S_k(\theta) - S_{k+12}(\theta) > 0
$$
for $\theta \in [\beta_{j+1}-\pi/3(k+12), \alpha_j]$.
Taking $\theta=\alpha_j^*$, we find that $G_{k+12}(\alpha_j^*)<0$.
On the other hand, $G_{k+12}(\beta_{j+1})>0$. Hence
$\alpha_j^* <  \beta_{j+1}^*$. This completes the proof of Case (a).

\subsubsection{\bf Proof of (iii) and (iv) when (b) occurs} In this case 
$\cos((k+12)\theta/2)$ is decreasing at $\beta_{j+1}$. Hence
\begin{equation}\label{103}
\alpha_j^* < \alpha_j \ \ \text{and} \ \ \beta_{j+1}^*< \beta_{j+1}.
\end{equation}
Further
\begin{equation}\label{104}
\beta_j< \beta_j^* \ \ \text{if} \ \ \beta_{j} \ge \frac{19\pi}{32} + \frac{\pi}{3(k+12)}.
\end{equation}

\begin{figure}[H]
\begin{center}
 \begin{minipage}[b]{0.4\textwidth}
  \includegraphics[width=5cm]{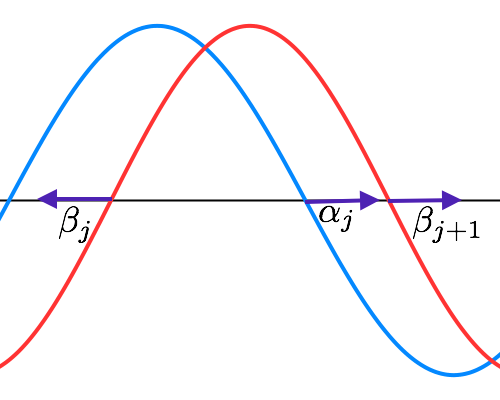}
  \label{p3}
   \end{minipage}
   \end{center}
\end{figure}
From \eqref{103} and \eqref{104} we easily get
$$
\beta_j^* < \alpha_j^* \ \ \text{if} \ \ \beta_{j} \ge \frac{19\pi}{32} + \frac{\pi}{3(k+12)},
$$
which proves (iii). Thus we proceed to prove (iv) i.e.
$\alpha_j^* <  \beta_{j+1}^*$. If $\alpha_j + \frac{\pi}{3k} \le \beta_{j+1}$,
then by \lemref{l1} and \eqref{103}, we have
$$
\alpha_j^* < \alpha_{j}+\frac{\pi}{3k} \le \beta_{j+1} < \beta_{j+1}^*.
$$
Hence we can assume that
$$
\beta_{j+1} < \alpha_j+\frac{\pi}{3k} \ \ \text{and also} \ \ 
\alpha_j^*, \beta_{j+1}^* \in \left(\beta_{j+1}, \alpha_j+\frac{\pi}{3k}\right).
$$
Note that $\beta_{j+1} < \alpha_j+\frac{\pi}{3k}$ implies
\begin{equation}\label{j-cond}
j > \begin{cases}
     (5k+24)/72 & \ \text{if} \ k \equiv 0 \bmod 4\\
     (5k-12)/72 & \ \text{if} \ k \equiv 2 \bmod 4.
    \end{cases}
\end{equation}

Here the aim is to show that
\begin{equation}\label{G-ineq}
G_{k+12}(\theta)>G_k(\theta) \ \ \text{for} \ \ \theta \in [\beta_{j+1}, \alpha_j+\pi/3k].
\end{equation}
If \eqref{G-ineq} holds, then $G_k(\beta_{j+1}^*)<0$. Since
$G_k(\alpha_j)>0$, we get $\alpha_j < \alpha_{j}^* < \beta_{j+1}^*$
as required. Thus it remains to prove \eqref{G-ineq}.
Note that
\begin{align*}
G_{k+12}(\theta) - G_k(\theta)=
& 2\cos((k+12)\theta /2) - 2\cos(k \theta /2)\\
&+ (2\cos(\theta /2))^{-(k+12)}-(2\cos(\theta /2))^{-k}
+ Q_{k+12}(\theta) - Q_k(\theta).
\end{align*}
Let $r(\theta):=2\cos((k+12)\theta /2) - 2\cos(k \theta /2)$.
We first find a lower bound for $r(\theta)$ in $[\beta_{j+1}, \alpha_j+\pi/3k]$.
For this we show that $r(\theta)$ is decreasing in this interval.
Note that $r(\theta)= -4\sin((k+6)\theta/2) \sin(3\theta)$.
Now $-\sin(3\theta)$ is decreasing from $1$ to $0$ in 
$(\pi/2,2\pi/3)$. Also $r(\theta)>0$ for $\theta \in (\gamma_j,\gamma_{j+1})$.
Thus $\sin((k+6)\theta/2)>0$ for $\theta \in (\gamma_j,\gamma_{j+1})$.
As $\gamma_j$ and $\gamma_{j+1}$ are two consecutive zeros of
$\sin((k+6)\theta/2)$, we therefore obtain that $\sin((k+6)\theta/2)$ is
decreasing in $\left(\frac{\gamma_j+\gamma_{j+1}}{2},\gamma_{j+1}\right)$.
From \eqref{alpha_j}, \eqref{beta_j} and \eqref{gamma_j}, we deduce that
$$
\frac{\gamma_j+\gamma_{j+1}}{2}<\beta_{j+1}
\ \ \text{and} \ \
\alpha_j+\frac{\pi}{3k} < \gamma_{j+1}.
$$
Hence $\sin((k+6)\theta/2)$ is decreasing in
$[\beta_{j+1}, \alpha_j+\pi/3k]$. As both the functions
$\sin((k+6)\theta/2)$ and $-\sin(3\theta)$ are positive and decreasing,
we get that $r(\theta)$ is decreasing in 
$[\beta_{j+1}, \alpha_j+\pi/3k]$.
Therefore $r(\theta)$ is minimised at $\alpha_j+\pi/3k$ for
$\theta \in [\beta_{j+1}, \alpha_j+\pi/3k]$.
Note that $2\cos\left(\frac{k}{2}\left(\alpha_j+\frac{\pi}{3k}\right)\right)=-1$.
Now
$$
\alpha_j+\frac{\pi}{3k} = \begin{cases}
     \beta_{j+1}+ \frac{72j-5k-24}{3k(k+12)}\pi & \ \text{if} \ k \equiv 0 \bmod 4,\\
     \beta_{j+1}+ \frac{72j-5k+12}{3k(k+12)}\pi & \ \text{if} \ k \equiv 2 \bmod 4.
    \end{cases}
$$
Hence
$$
\left| 2\cos\left(\frac{k+12}{2}\left(\alpha_j+\frac{\pi}{3k}\right)\right) \right| = \begin{cases}
     \left| 2\sin\left(\frac{72j-5k-24}{6k}\pi\right)\right| & \ \text{if} \ k \equiv 0 \bmod 4,\\
     \left| 2\sin\left(\frac{72j-5k+12}{6k}\pi\right)\right| & \ \text{if} \ k \equiv 2 \bmod 4.
    \end{cases}
$$
From \eqref{j-cond} and the condition $j \le (k-s)/12$, we get that
if $k \equiv 0 \bmod 4$,
$$
0<\frac{72j-5k-24}{6k}\le\frac{k-6s-24}{6k}\le\frac{1}{6}-\frac{4}{k}
$$
and if $k \equiv 2 \bmod 4$,
$$
0<\frac{72j-5k+12}{6k}\le\frac{k-6s+12}{6k}\le\frac{1}{6}-\frac{4}{k}.
$$
Therefore,
$$
\left| 2\cos\left(\frac{k+12}{2}\left(\alpha_j+\frac{\pi}{3k}\right)\right) \right|
\le 2\sin\left(\frac{\pi}{6}-\frac{4\pi}{k}\right).
$$
The equation of the tangent line of the sine curve at $\pi/6$ is
$$
y=\frac{1}{2} + \frac{\sqrt 3}{2} \left(x-\frac{\pi}{6}\right).
$$
Hence $2\sin\left(\frac{\pi}{6}-\frac{4\pi}{k}\right) < 1-\frac{4 \sqrt 3 \pi}{k}$.

Next let $t(\theta):=(2\cos(\theta /2))^{-(k+12)}-(2\cos(\theta /2))^{-k}$. Then
$$
t'(\theta)=\sin(\theta/2)(2\cos(\theta /2))^{-(k+1)}
\left( (k+12)(2\cos(\theta /2))^{-12}-k \right).
$$
So $t(\theta)$ is minimised when $(2\cos(\theta /2))^{-12}=\frac{k}{k+12}$.
We thus get,
\begin{align*}
G_{k+12}(\theta) - G_k(\theta)
& > 1 - \left(1-\frac{4 \sqrt 3 \pi}{k}\right)
+ \left(\frac{k}{k+12}\right)^{k/12} \left(\frac{k}{k+12}-1\right)
- 42.718 \left(\frac{1}{2}\right)^{k/2}\\
& > \frac{4 \sqrt 3 \pi}{k} - \frac{12}{k+12} - 42.718 \left(\frac{1}{2}\right)^{k/2}\\
& = \left(\frac{12}{k} -\frac{12}{k+12} \right) + \frac{4 \sqrt 3 \pi-12}{k} - 42.718 \left(\frac{1}{2}\right)^{k/2}
> 0.
\end{align*}
This completes the proof of Case (b). 

\subsection{Remaining cases}
By Cases I and II, in order to complete the proof of \thmref{interlacing},
we need to prove the following two cases.
\begin{itemize}
\item[(v)] $\alpha_j^* < \beta_{j+1}^*$ when
$\alpha_j \in \left( \frac{\pi}{2}, \frac{23\pi}{36}-\frac{6 \pi}{k(k+12)} \right]$ and
$\beta_{j+1}+ \frac{6 \pi}{(k+12)(k+24)} > \frac{23\pi}{36}$.
\item[(vi)] $\beta_{j}^* < \alpha_j^*$ when
$\alpha_j \in \left[ \frac{19\pi}{32} + \frac{\pi}{3k},\frac{2\pi}{3} \right)$ and
$\beta_{j} < \frac{19\pi}{32} + \frac{\pi}{3(k+12)}$.
\end{itemize}
\subsubsection{\bf Proof of (v)}
We show that  $\alpha_j \ge \frac{19\pi}{32} + \frac{\pi}{3k}$. Hence by \S4.2,
$\alpha_j^* < \beta_{j+1}^*$. Suppose that $\alpha_j < \frac{19\pi}{32} + \frac{\pi}{3k}$.
Then using \eqref{alpha_j} we get that
\begin{equation}\label{1st-cond}
27k>
\begin{cases}
576j - 384 & \text{if} \ k \equiv 0 \bmod 4,\\
576j - 96 & \text{if} \ k \equiv 2 \bmod 4.
\end{cases}
\end{equation}
Since $\beta_{j+1}+ \frac{6 \pi}{(k+12)(k+24)} > \frac{23\pi}{36}$,
it follows from \rmkref{k-cond} and \eqref{beta_j} that $k \ge 24$ and
\begin{equation}\label{2nd-cond}
0<
\begin{cases}
(k+24)(72j-5k-24) + 216 & \text{if} \ k \equiv 0 \bmod 4,\\
(k+24)(72j-5k+12) + 216 & \text{if} \ k \equiv 2 \bmod 4
\end{cases}
\end{equation}
respectively. Using \eqref{1st-cond} in \eqref{2nd-cond} we get that
$$
0<
\begin{cases}
(k+24)(-13k+192) + 1728 & \text{if} \ k \equiv 0 \bmod 4,\\
(k+24)(-13k) + 1728 & \text{if} \ k \equiv 2 \bmod 4.
\end{cases}
$$
This is a contradiction since $k \ge 24$. 

\subsubsection{\bf Proof of (vi)}
As in (v), we can deduce that
$\alpha_j+ \frac{6 \pi}{k(k+12)} \le \frac{23\pi}{36}$.
Hence $\beta_j^* < \alpha_j^*$ by \S4.1. \qed

\medskip

\noindent {\bf Acknowledgement:}
We would like to thank Biswajyoti Saha and Jhansi Bhavani V. for helping us with 
the computations in \lemref{Eisenstein}.


\begin{thebibliography}{100}

\bibitem{AKN}
T. Asai, M. Kaneko and H. Ninomiya,
Zeros of certain modular functions and an application,
{\it Comment. Math. Univ. St. Paul.} {\bf 46} (1997), no. 1, 93--101. 

\bibitem{DJ}
W. Duke and P. Jenkins, 
On the zeros and coefficients of certain weakly holomorphic modular forms,
{\it Pure Appl. Math. Q.} {\bf 4} (2008), no. 4,
Special Issue: In honor of Jean-Pierre Serre, Part 1, 1327--1340.

\bibitem{EUG}
E.-U. Gekeler, Some observations on the arithmetic of Eisenstein
series for the modular group SL$(2,\Z)$, {\it Arch. Math. (Basel)}
{\bf 77} (2001), 5--21.

\bibitem{JG}
J. Getz, A generalization of a theorem of Rankin and Swinnerton-Dyer
on zeros of modular forms, {\it Proc. Amer. Math. Soc.} {\bf132} (2004),
no. 8, 2221--2231.

\bibitem{JP}
P. Jenkins and K. Pratt, Interlacing of zeros of weakly holomorphic
modular forms, {\it Proc. Amer. Math. Soc. Ser. B} {\bf 1} (2014), 63--77. 

\bibitem{JJ}
J. Jermann, Interlacing property of the zeros of $j_n(\tau)$, {\it Proc. Amer.
Math. Soc.} {\bf 140} (2012), no. 10, 3385--3396.

\bibitem{HN}
H. Nozaki, A separation property of the zeros of Eisenstein series for
SL$(2,\Z)$, {\it Bull. Lond. Math. Soc.} {\bf 40} (2008), no. 1, 26--36.

\bibitem{RSD}
F. K. C. Rankin and H. P. F. Swinnerton-Dyer, On the zeros of Eisenstein
series, {\it Bull. London Math. Soc.} {\bf 2} (1970), 169--170.

\bibitem{RAR}
R. A. Rankin, The zeros of certain Poincar\'e series, 
{\it Compositio Math.} {\bf 46} (1982), no. 3, 255--272.

\end{thebibliography}
\end{document}